\documentclass{amsart}

\usepackage{color}
\usepackage{mathtools}
\usepackage{commath}
\usepackage{hyperref}

\newtheorem{theorem}{Theorem}[section]
\newtheorem{lemma}[theorem]{Lemma}
\newtheorem{proposition}[theorem]{Proposition}
\newtheorem{corollary}[theorem]{Corollary}

\theoremstyle{definition}
\newtheorem{definition}[theorem]{Definition}

\theoremstyle{remark}
\newtheorem{remark}[theorem]{Remark}

\numberwithin{equation}{section}

\newcommand{\R}{\mathbb{R}}
\newcommand{\N}{\mathbb{N}}
\newcommand{\Z}{\mathbb{Z}}
\newcommand{\cC}{\mathcal{C}}
\newcommand{\cQ}{\mathcal{Q}}
\newcommand{\cH}{\mathcal{H}}
\newcommand{\cP}{\mathcal{P}}
\newcommand{\restr}{\mathbin{\vrule height 1.6ex depth 0pt width
0.13ex \vrule height 0.13ex depth 0pt width 1.3ex}}
\newcommand{\dist}{\operatorname{dist}}
\newcommand{\diam}{\operatorname{diam}}
\newcommand{\lip}{\operatorname{Lip}}
\newcommand{\defeq}{\vcentcolon=}

\begin{document}
\title{Wild examples of rectifiable sets}

\author{Max Goering}

\author{Sean McCurdy}

\thanks{The first author was partially supported by NSF grants DMS-1664867 and DMS-1500098.
The second author was partially supported by T. Toro's Craig McKibben \& Sarah Merner Professor in Mathematics. }


\begin{abstract}
We study the geometry of sets based on the behavior of the Jones function, $J_{E}(x) = \int_{0}^{1} \beta_{E;2}^{1}(x,r)^{2} \frac{dr}{r}$.  We construct two examples of countably $1$-rectifiable sets in $\R^{2}$ with positive and finite $\mathcal{H}^1$-measure for which the Jones function is nowhere locally integrable. These examples satisfy different regularity properties: one is connected and one is Ahlfors regular. Both examples can be generalized to higher-dimension and co-dimension. 
\end{abstract}

\maketitle

\section{Introduction}
In his solution to the Analyst's Traveling Salesman Problem \cite{jones1990rectifiable}, Peter Jones introduced a local gauge of flatness which has been generalized by David and Semmes \cite{david1991singular} to measures and higher dimensions.  These families of local gauges of flatness are called the Jones $\beta$-numbers, and they have come to dominate the landscape in quantitative techniques relating rectifiability, potential theory, and boundedness of singular integrals.  See, for example the landmark book \cite{david1993analysis}.

For a set $E \subset \R^d$, $1\le p < \infty$, and an integer $1 \le n \le d-1$, we write $\mu = \cH^{n} \restr E$ and define the Jones $\beta$-numbers as follows,
$$
\beta_{E;p}^{n}(x,r)= \left(\inf_{L \subset \R^{d} \text{ an $n$-plane}} \int_{B(x,r)} \left( \frac{ \dist(y,L)}{r} \right)^{p} \frac{ d \mu(y)}{r^{n}} \right)^{\frac{1}{p}}.
$$

We also write $\beta_{\mu;p}^{n}(x,r)$ for $\beta_{E;p}^{n}(x,r)$, when $\mu = \mathcal{H}^n \restr E$ is understood. If $p = \infty$, the $\beta$-numbers are defined in terms of the $\sup$-norm instead of the $L^{p}$-norm.

In addition to generalizing the Jones $\beta$-numbers, \cite{david1991singular} also introduced the notion of uniform rectifiability. A set $E \subset \R^{d}$ is said to be Ahlfors $n$-regular if there exists $0 < c < C < \infty$ such that $cr^{n} \le \cH^{n}(E \cap B(x,r)) \le C r^{n}$ for all $x \in E$ and all $0 < r < \diam(E)$. An $n$-Ahlfors regular $E \subset \R^{d}$ is said to be uniformly $n$-rectifiable if there exist finite constants $\theta, \Lambda > 0$ such that for all $x \in E$ and all $0 < r < \diam(E)$ there is a Lipschitz mapping $g : B(0,r) \subset \R^{n} \to \R^{d}$ with $\lip(g) \le \Lambda$ such that $\cH^{n}(E \cap B(x,r) \cap g(B(0,r))) \ge \theta r^{n}$.

In \cite{david1991singular} the authors show that an $n$-Ahlfors regular set $E\subset \R^d$, is $n$-uniformly rectifiable if and only if the Jones $\beta$-numbers satisfy the  following Carleson condition for some $1 \le p < \frac{2n}{n-2}$, 
\begin{equation} \label{e:betapcarleson}
C_{E;p}^{n}(x,R) \defeq \int_{B(x,R)} \int_{0}^{R} \beta_{E;p}^{n}(y,r)^{2} \frac{ \dif r}{r} \dif \mu(y) \le c R^{n} \quad \text{for all } x \in E, ~ R > 0
\end{equation}

A set $E \subset \R^{d}$ is said to be countably $n$-rectifiable if there are Lipschitz maps $f_{i} : \R^{n} \to \R^{d}$ with $i = 1, 2, \dots, $ such that
\begin{equation*} 
\cH^{n} ( \R^{d} \setminus \cup_{i} f_{i}(\R^{n}) ) = 0.
\end{equation*}

Recently, Tolsa \cite{tolsa2015characterization} and Azzam and Tolsa \cite{azzam2015characterization} show as a special case of their results that $E$ is countably $n$-rectifiable if and only if 
\begin{equation} \label{e:fJf}
J_{E}^{n}(x,1) = \int_{0}^{1} \beta_{E;2}^{n}(x,r)^{2} \frac{ \dif r}{r}  < \infty \quad \text{for } \mathcal{H}^n-a.e.~ x \in E
\end{equation}
where $J_{E}^{n}(x,1)$ is the Jones function at $x$ and scale $1$. See also \cite{pajot1997conditions} and \cite{badger2016two} .

In this paper, we show that sets which satisfy \eqref{e:fJf} can fail to satisfy \eqref{e:betapcarleson} as dramatically as possible. 



\begin{theorem}\label{t:K_0}
There exists a rectifiable curve (of finite length), $K_0 \subset \R^2$, such that for $\mu = \cH^{1} \restr K_{0}$, for any $x \in K_0$, and any $\delta >0$
\begin{equation*}
\int_{B_{\delta}(x)} \int_{0}^{\delta} \beta_{\mu;2}^{1}(y,r)^{2} \frac{\dif r}{r}\dif \mu(y) = \infty 
\end{equation*}
\end{theorem}

The set $K_0$ arises from unions of modifications of approximations to snowflake-like sets.  
Note that by the Analyst's Traveling Salesman theorem \cite{jones1990rectifiable}
$$
\int_{\R^2} \int_0^\infty \beta_{\mu, \infty}^1(y, r)^{2} \frac{\dif r}{r}d \mu(y) < \infty,
$$
which prevents $K_{0}$ from being upper regular at a generic point.

\begin{theorem}\label{t:A_0}
There is a $1$-Ahlfors regular, countably $1$-rectifiable set $A_{0}$ contained in the unit cube in $\R^{2}$ such that for $\mu = \cH^{1} \restr A_{0}$, for every $x \in A_0$, and for every $\delta>0$, 
\begin{equation*}
\int_{B_{\delta}(x)} \int_{0}^{\delta} \beta_{A_0;2}^{1}(y,r)^{2} \frac{\dif r}{r}\dif \mu(y)= \infty.
\end{equation*}
\end{theorem}
The set $A_{0}$, whose construction was initially motivated by the machinery introduced in \cite{tolsa2017rectifiability}, is created from scaled unions of approximations to the $4$-corner Cantor set. Ultimately the presentation was simpler using the framework of self-similar sets.

\begin{remark} 
These examples can be used to create higher-dimensional ones by taking Cartesian products with finite intervals. That is, if $A \in \{K_{0}, A_{0}\}$ for any positive integer $n<d$, define $E^{\prime} = A \times [0,1]^{n-1} \subset \R^{n+1}$. Embedding $E^{\prime}$ into the first $(n+1)$-dimensions of $\R^{d}$ preserves the properties of $A$. In particular, it is standard that defining $\beta$-numbers over cubes (with sides parallel to the axes in $\R^{d}$) instead of balls leads to an equivalent definition of the $\beta$-numbers. Consequently finiteness of $C^{n}_{E;2}(x,R)$ is equivalent to the finiteness of $C^{1}_{A;2}(x^{\prime},R)$ where $x^{\prime}$ is the orthogonal projection of $x$ into $\R^{2}$.
\end{remark}

\section{Proof of Theorem \ref{t:K_0}}
For the remainder of this paper, we only consider $E \subset \R^{2}$ and the $\beta$-numbers when $p=2$.  As such, we write $\beta_{E}, \beta_{\mu}$, $C_{E}$, and $C_{\mu}$ in place of $\beta_{E;2}^{1}, \beta_{\mu;2}^{1}$, $C_{E;2}^{1}$, and $C_{\mu;2}^{1}$. Moreover, for any set $L \subset \R^{2}$ we write $B_{r}(L) = \{ x : \dist(x, L) < r \}$ and $B_{r} = B_{r}(\{0\})$.

We begin by stating two basic properties of the Jones $\beta-$numbers.  The first controls how fast the $\beta-$numbers can shrink by relating the $\beta-$numbers at comparable scales.  This property is often called ``doubling," though we have chosen to scale by the number $3$.  The second property shows how the $\beta-$numbers behave under rescaling.

\begin{proposition}\label{f:02}
Let $E\subset \R^2$ have $dim_{\mathcal{H}}(E) = 1$.
\begin{enumerate}
\item For any ball $B_r(y) \subset B_{3r}(x)$, 
$$
\beta_{E}(y,r)^{2} \le 3 \beta_{E}(x,3r)^{2}
$$
\item The $\beta$-numbers have the following scaling property. If $E^{z,t} = tE + z$ then
$
\beta_{E^{z,t}}(x,r)^{2} = \beta_{E} \left( \frac{ x-z}{t}, \frac{r}{t} \right)^{2}.
$
Consequently, $C_{E^{z,t}}(z,r) = t C_{E}(0, t^{-1} r)$.

\end{enumerate}
\end{proposition}

To construct a $1$-rectifiable set that is connected (hence Ahlfors lower-regular) for which the Jones function is locally non-integrable, we modify approximations to the Koch snowflake. This set will not be upper regular. Recall some facts about the standard approximation to the Koch snowflake.  

\begin{definition}\label{d: P}
Let $I \subset \R^2$ be a line segment, and fix $0< \alpha < \pi/2$. Define $P(I)$ as the set which results from the following operation
\begin{itemize}
\item[1.] Divide $I$ into three equal subintervals, $I_{\textrm{left}} \cup I_{\textrm{center}} \cup I_{\textrm{right}}$. 
\item[2.] Over the middle interval, $I_{\textrm{center}}$, construct an isosceles triangle with angles $\alpha$ and base $I_{\textrm{center}}$. 
\item[3.] Delete $I_{\textrm{center}}$, the base of the isosceles triangle.
\end{itemize}

We define 
\begin{equation} \label{e:bump}
S(I) =  \overline{P(I) \setminus I},
\end{equation}
and call $S(I)$ the \emph{bump}. If $q_{I}$ is the orthogonal projection onto the line containing $I$ and $q_{I}^{\perp}$ is the orthogonal projection onto $I^{\perp}$, then $\text{height}(S(I)) = \diam \{ q_{I}^{\perp} (S(I))\}$ and $\text{width}{(S(I))} = \diam \{\pi_{I}(S(I))\} = \frac{1}{3} \cH^{1}(I)$. We shall abuse our notation slightly by saying that for a collection of line segments, $E$, the set $P(E)$ is obtained by applying $P$ to each maximal line segment contained in $E$. 
\end{definition}

If $I = [0, 1] \times \{0\}$ and $\alpha = \frac{\pi}{3}$, the standard approximations to the Koch snowflake are given by $\{P^k(I)\}_{k=1}^{\infty}$, where $P^{k}$ denotes applying $P$ iteratively $k$ times. We emphasize a few properties about deformations under the operation $P$.

\begin{proposition} \label{f:01} For any finite line segment $I \subset \R^{2}$ and positive integer $n$,
\begin{align}
\label{e:ha}  \mathrm{height}(S(I)) &= \frac{\tan(\alpha)}{6} |I| \\
\label{e:sa}  \cH^{1}(S(I)) &= \frac{\sec(\alpha)}{3} |I|  \\
\label{e:pa}  \cH^{1}(P^{n}(E)) &= \left( \frac{\sec(\alpha)+2}{3} \right)^{n} \cH^{1}(E)
\end{align}
When $\tau = \frac{1}{20} \min \left\{ \frac{ \tan(\alpha)}{6}, \frac{1}{3} \right\}$, there exists $c_{0}= c(\alpha)$ such that for all lines $L$
\begin{equation}
\label{e:wma}  \cH^{1} \left( S(I) \setminus B_{\tau}(L) \right) \ge c_{0} \cH^{1}(S(I)) .
\end{equation}
\end{proposition}

\begin{proof}
\eqref{e:ha} and \eqref{e:sa} follow from planar geometry. The $n=1$ case for \eqref{e:pa} follows by adding back in the unchanged intervals $I_{\text{left}}$ and $I_{\text{right}}$, which have total length $\frac{2}{3} |I|$. The geometric nature of the definition of $P$ allows us to then iterate this to achieve \eqref{e:pa}.

To verify \eqref{e:wma} we proceed by contradiction. Suppose no such constant $c_{0}$ exists. Then, there exists a sequence of lines intersecting $S(I)$ such that 
$$
\cH^{1} \left( S(I) \setminus B_{\tau}(L_{i}) \right) < 2^{-i} \cH^{1}(S(I)).
$$ 
After passing to a subsequence, $L_{i}$ converge to some line $L$ with the property that $\cH^{1} \left( S(I) \setminus B_{\tau}(L) \right) = 0$. Since $S(I)$ is connected, this implies $S(I) \subset B_{2 \tau}(L)$. However, this contradicts the fact that $2 \tau \le \frac{1}{10} \min \{\mathrm{height}(S(I)) , \mathrm{width}(S(I))\}$.
\end{proof}

\begin{definition}
Define $\cP_{j}$ to be the set operation defined on line-segments by
\begin{equation*} 
\cP_{j}(I) = P^{j-1}(S(I)) \bigcup \left( I \setminus I_{\text{center}} \right),
\end{equation*}
recalling the definition of $S(I)$ can be found in \eqref{e:bump}. Loosely speaking, for any line segment, $I$, $\cP_{j}(I)$ is the set that replaces the center of $I$ with a $j$th approximation of the Koch curve.
\end{definition}

\begin{corollary}
For any line segment $I \subset \R^{2}$ and positive integer $n$ 
\begin{equation} \label{e:cpa}
\cH^{1}\left( \cP_{n}(I) \right) = \frac{2}{3} |I| + \left( \frac{ \sec(\alpha) + 2}{3} \right)^{n-1} \frac{\sec(\alpha)}{3} |I|.
\end{equation}

Moreover, if $\alpha \le \pi/3$, 
\begin{equation} 
\label{e:hdist} \dist_{\cH}(I, P^{n}(I)) \le \frac{\tan(\alpha)}{12}|I|.
\end{equation}
\end{corollary}

\begin{proof}
Equations \eqref{e:sa} and \eqref{e:pa} verify \eqref{e:cpa}. Indeed,
$$ \hspace{-.4in}
 \cH^{1}\left(P^{n-1}(S(I)) \right) = \left( \frac{ \sec(\alpha) + 2}{3} \right)^{n-1} \cH^{1}(S(I)) = \left( \frac{ \sec(\alpha) + 2}{3} \right)^{n-1} \frac{\sec(\alpha)}{3} |I|.
$$

The restriction to $\alpha \le \pi/3$ ensures the longest line segment of $P^{i}(I)$ has length at most $3^{-i}$. Consequently, \eqref{e:ha} guarantees
\begin{equation*}
\dist_{\cH}(P^{n}(I), I) \le \sum_{i=1}^{n} \dist_{\cH}(P^{i}(I), P^{i-1}(I)) \le \sum_{i=1}^{n} 3^{-i} \mathrm{height}(S(I))  \le \frac{\tan(\alpha)}{12}|I| .
\end{equation*}
\end{proof}

\begin{definition}
Now, we let $n$ be a natural number to be chosen later and $E_0 = I = [0, 1] \times \{0\}$.  We define $E_{1} = \cP_{n}(I)$. For $k \ge 2$ inductively define 
\begin{equation}\label{d:E_k}
E_{k} = \cP_{kn} \left( \left[0, 3^{-(k-1)} \right] \times \{0\} \right) \bigcup \left( \left\{ \left[ 3^{-(k-1)}, 1 \right] \times \R \right\} \cap E_{k-1} \right).
\end{equation}
\end{definition}
Notably, for all integers $j$ the operation $\cP_{j}$ applied to $[0, 3^{-(k-1)}] \times \{0\}$ leaves the segment $[0, 3^{-k}] \times \{0\}$ untouched. Consequently, the sequence of sets $\{E_{k}\}$ are defined by replacing the ``next'' triadic interval with a scaled approximation of the Koch snowflake. The fact that each triadic strip $[3^{-k}, 3^{-(k-1)}] \times \R$ is only modified once in the sequence of sets $E_{k}$ is ensures the Hausdorff dimension of the final set remains $1$.

\begin{lemma}[Base Set] \label{l:01}
Fix $\alpha \le \pi/3$ and any integer $n$ satisfying\footnote{Note that for instance, $\alpha = \pi/3$ and $n \in \{2, 3\}$ satisfies \eqref{e:nbounds}.}
\begin{equation} \label{e:nbounds}
3^{-1} \left( \frac{ \sec(\alpha) +2}{3} \right)^{n} < 1 < 3^{-1} \left( \frac{ \sec(\alpha) + 2}{3} \right)^{2n}. \end{equation}
Then the sequence of sets $E_{k}$ from \eqref{d:E_k} converge to a compact and connected Borel set $E_{\infty}$ in the Hausdorff topology on compact subsets. Furthermore, $E_{\infty}$ satisfies:

\begin{enumerate}
\item $\mathcal{H}^1(E_{\infty}) < \infty $ 
\item For all $\delta >0$, $C_{E_{\infty}}(0, \delta) = + \infty$.
\end{enumerate}
\end{lemma}

\begin{proof}
The existence of the limiting compact set $E_{\infty}$ follows from precompactness of sets contained in $B_{10}$ in the Hausdorff distance and \eqref{e:hdist} which ensures that $\dist_{\cH}(E_{k+1},E_{k}) \sim 3^{-k}$. Connected follows since $E_{\infty} \setminus B_{3^{-k}}(0)=E_{k} \setminus B_{3^{-k}}(0)$ is connected for each $k$.

To see that $E_{\infty}$ has finite length we write $\cH^{1}(E_{k}) = \cH^{1}(E_{k} \setminus B_{3^{1-k}}) + \cH^{1}(E_{k} \cap B_{3^{1-k}})$. Since $E_{k} \setminus B_{3^{1-k}} = E_{k-1} \setminus B_{3^{1-k}}$ and $\cH^{1}(E_{k-1} \setminus B_{3^{1-k}}) = \cH^{1}(E_{k-1}) - 3^{1-k}$, \eqref{e:cpa} implies,

$$
\cH^{1}(E_{k}) - \cH^{1}(E_{k-1}) = 3^{-k} \left[ \left( \frac{ \sec(\alpha) + 2}{3} \right)^{nk} \sec(\alpha) - 1 \right].
$$
Since $\cH^{1}(E) = 1$, iteration yields
\begin{equation} \label{e:massek}
\cH^{1}(E_{k}) = 1 + \sum_{i=1}^{k} 3^{-i} \left[ \left( \frac{ \sec(\alpha) + 2}{3} \right)^{ni} \sec(\alpha) -1 \right].
\end{equation} 
In particular, $\lim_{k \to \infty} \cH^{1}(E_{k}) < \infty$ whenever $n$ satisfies the lower bound from \eqref{e:nbounds}. Moreover $\cH^{1}(E_{\infty}) = \lim_{k \to \infty} \cH^{1}(E_{k})$ since for all $j \ge k$, 
$$
\cH^{1} \left( E_{j} \Delta E_{k} \right) \le 2 \sum_{i=k+1}^{\infty} 3^{-i} \left( \frac{ \sec(\alpha) + 2}{3} \right)^{ni} \sec(\alpha) ,
$$
which decays to zero as $k \to \infty$. Hence, \eqref{e:massek} holds for $E_{\infty}$  and $0 < \cH^{1}(E_{\infty})<\infty$. 

It only remains to show $C_{E_{\infty}}(0,\delta) = + \infty$ for all $\delta > 0$. To this end, we first note that when $r = r(n,\alpha) = 3^{-1} \left( \frac{ \sec(\alpha) + 2}{3} \right)^{n}$,
\begin{equation} \label{e:kmb}
\cH^{1}(E_{\infty} \cap B_{3^{-k}}(0)) = 3^{-k} + \sec(\alpha)\frac{r^{k+1}}{1 - r} - \frac{3^{-(k+1)}}{1 - 3^{-1}}.
\end{equation}

Indeed, by \eqref{e:massek} and the trick used to prove \eqref{e:massek}
\begin{align*}
\cH^{1}(E_{\infty} \cap B_{3^{-k}}(0)) 
& = 3^{-k} + \sum_{i=k+1}^{\infty} 3^{-i} \left[ \sec(\alpha) \left( \frac{\sec(\alpha) + 2}{3} \right)^{ni} -1 \right].
\end{align*}
\underline{Claim:} With $\tau$ as in Proposition \ref{f:01} and $\alpha \le \pi/3$, there exists a constant $c_{1}$ and integer $j_{0}$ independent of $k$ such that for any line $L$, and all $k$ such that $nk - 1 - j_{0} \ge 0$,
\begin{equation} \label{e:c2}
\cH^{1} \left(\left( E_{\infty} \setminus B_{\frac{\tau}{2\cdot 3^{k}}}(L) \right) \cap B_{3^{-k}} \right) \ge c_{1} 3^{-k} \left( \frac{\sec(\alpha)+2}{3} \right)^{nk-1-j_{0}}.
\end{equation}

\emph{Proof of Claim.} Writing $I^{\prime} = [0,1] \times \{0\}$, we will in fact scale by $3^{k}$ and show the stronger result that
\begin{equation*}
\cH^{1} \left( \left(\cP_{nk}(I^{\prime}) \setminus B_{\frac{\tau}{2 \cdot 3^{0}}}(L)\right) \cap B_{3^{0}}  \right) \ge c_{1} 3^{0} \left( \frac{\sec(\alpha)+2}{3} \right)^{nk-1-j_{0}} |I^{\prime}|.
\end{equation*}

To do so, we find a line segment $J \subset S(I^{\prime}) \setminus B_{\tau}(L)$ such that $J$ has an endpoint in common with one of the two line segments of $S(I^{\prime})$ and $|J| = 3^{-j_{0}}\cH^{1}(S(I^{\prime}))/2$, where $j_{0}$ to be chosen later is independent of $L$. This specific choice of length and endpoint ensure that $P^{nk-1-j_{0}}(J) \subset \cP_{nk}(I^{\prime})$. Moreover, the choice of $j_{0}$ will both guarantee that $|J|$ is large enough and that $P^{nk-1-j_{0}}(J)$ remains outside of $B_{\tau/2}(L)$, hence verifying Claim 2. 

To find $J$, we note that the simple shape of $S(I^{\prime})$ guarantees that $S(I^{\prime}) \setminus B_{\tau}(L)$ has at most 4 maximal line segments. Hence, there exists a maximal line segment $K_{L} \subset S(I^{\prime}) \setminus B_{\tau}(L)$ with $\cH^{1}(K_{L}) \ge \frac{1}{4} \cH^{1} \left( S(I^{\prime}) \setminus B_{\tau}(L) \right)$. If $K_{L}$ is parallel to $L$ let $x_{L}$ denote either endpoint of $K_{L}$. Otherwise, let $x_{L}$ denote the unique endpoint of $K_{L}$ that is not contained in $\overline{B_{\tau}(L)}$. Define $J$ to be the unique subset of $K_{L}$ of length $3^{-j_{0}} \frac{\sec(\alpha)}{6} |I^{\prime}|$ with endpoint $x_{L}$. Now, define $j_{0}$ as the smallest integer such that 
$$
3^{-j_{0}} < \min\left\{ \frac{c_{0}}{4}, \left(\frac{\tan(\alpha)}{12} \cdot \frac{\sec(\alpha)}{6} |I^{\prime}| \right)^{-1} \frac{\tau}{2} \right\},
$$
where $c_{0}$ is as in Proposition \ref{f:01}.  The first condition ensures that $J \subset K_{L}$ and \eqref{e:wma} guarantees that the first constraint on $j_{0}$ is independent of $L$ and $k$. The second constraint combined with \eqref{e:sa} and \eqref{e:hdist} ensure that $\dist_{\cH}(P^{nk-1-j_{0}}(J),J) \le \frac{\tau}{2}$. Moreover, choosing $j_{0}$ to be the smallest admissible integer, and guarantees that $|J| = 3^{-j_{0}} \frac{\sec(\alpha)}{6} |I^{\prime}| \ge c^{\prime} |I^{\prime}|$ where $c^{\prime}$ is independent of $L$ and $k$. Finally, \eqref{e:pa} completes the proof of the Claim since
$$
\cH^{1} \left( \cP_{nk}(I^{\prime}) \setminus B_{\tau/2}(L) \right) \ge \cH^{1}(P^{nk-1-j_{0}}(J)) \ge c_{1}\left( \frac{\sec(\alpha)+2}{3} \right)^{nk-1-j_{0}} |I^{\prime}|,
$$
where $c_{1}$ depends only on $\alpha$.

Whenever $nk-1-j_{0} \ge 0$, \eqref{e:c2} implies 
\begin{equation} \label{e:einfbeta}
\beta_{E_{\infty}}(0,3^{-k})^{2} \ge \frac{1}{3^{-k}} \left( \frac{ \frac{\tau}{2 \cdot 3^{k}}}{3^{-k}} \right)^{2} \left( c_{1} 3^{-k} \left( \frac{\sec(\alpha)+2}{3} \right)^{nk-1-j_{0}} \right) = c_{2} \left( \frac{\sec(\alpha)+2}{3} \right)^{nk}
\end{equation}

Fix $\delta > 0$ and any integer $k_{\delta}$ such that $3^{-k_{\delta}} < \delta$ and $n k_{\delta} - 1 - j_{0} \ge 0$. Then, with $\mu = \cH^{1} \restr E_{\infty}$, repeated applications of Proposition \ref{f:02}, \eqref{e:einfbeta}, and \eqref{e:kmb} yield
\begin{align*}
\int_{B_{\delta}(0)} &\int_{0}^{\delta} \beta_{\mu}(x,r)^{2} \frac{dr}{r} d \mu(x) 
 \ge \ln(3) 3^{-2} \sum_{k=k_{\delta}}^{\infty} \mu(B_{3^{-(k+2)}}) \beta_{\mu}(0,3^{-(k+2)})^2 \\
& \ge \ln(3) 3^{-2} \sum_{k=k_{\delta}}^{\infty} \left( 3^{-k} +  \sec(\alpha) \frac{r^{k+1}}{1 - r} - \frac{3^{-(k+1)}}{1 - 3^{-1}} \right) \left(  c_{2}  \left( \frac{2+ \sec(\alpha)}{3} \right)^{nk} \right).
\end{align*}
Due to the lower bound in \eqref{e:nbounds}, this sum diverges if and only if
\begin{align*}
\sum_{k=k_{\delta}}^{\infty} \left[\sec(\alpha) \frac{r^{k+1}}{1- r} - \frac{1}{3^{k+1} - 3^{k}}  \right]  \left( \frac{2+ \sec(\alpha)}{3} \right)^{nk} = \sum_{k=k_{\delta}}^{\infty} \left[ \sec(\alpha) \frac{3^{k} r^{2k+1}}{1-r} - \frac{r^{k}}{3-1}\right]
\end{align*}
diverges. Since the lower bound in \eqref{e:nbounds} ensures $r <1$, this diverges if and only if $\sum_{k=k_{\delta}}^{\infty} (3 r^{2})^{k}$ diverges which is equivalent to the upper bound in \eqref{e:nbounds}. 
\end{proof}



\begin{theorem}\label{t:koch}
There exists a connected set, $K_0 \subset \R^2$ of finite $\cH^{1}$-measure such that for any $x \in K_0$ and $\delta >0$
\begin{equation*}
C_{K_0}(x, \delta)= \infty.
\end{equation*}
\end{theorem}

\begin{proof}[Proof of Theorem \ref{t:K_0}]. 
Let $\{r_i\}_{i = 1}^{\infty}$ be a sequence of positive numbers such that $\sum_i r_i \le 1$.  Let $E^{x,r} \subset \R^2$ be the set $E^{x, r} = rE_{\infty} +x$.  We construct $K_0$ as the union of a countable collection of nested sets $\{\Gamma_i\}$.  

Let $\Gamma_0 = E_{\infty}$.  Now, let $\{x_{1, j}\}_{j = 1}^{N_1}$ be a maximal $2^{-1 -1}$-separated collection of points in $\Gamma_0$.  Let 
\begin{align*}
\Gamma_1 = \Gamma_0 \cup \bigcup_{j=1}^{N_1} E^{ x_{1, j}, \frac{r_1}{N_1}}.  
\end{align*}

Suppose that we have defined $\Gamma_{i-1}$, some positive integers $\{N_{\ell}\}_{\ell=1}^{i-1}$ and a collection of points $\{x_{\ell,j} \in \Gamma_{i-2} \mid 1 \le \ell \le i-1, 1 \le j \le N_{\ell} \}$ that form a maximal $2^{-(i-1)-1}$ net for $\Gamma_{i-2}$.  Then choose $N_{i} \in \N$ and points $\{x_{i, j}\}_{1 \le j \le N_{i}} \subset \Gamma_{i-1}$ so that $\{x_{\ell,j} \in \Gamma_{i-1} \mid 1 \le \ell \le  i, 1 \le j \le N_{\ell}\}$ is a maximal $2^{-i -1}$ net in $\Gamma_{i-1}.$ Then define $\Gamma_{i}$ by
\begin{align*}
\Gamma_i = \Gamma_{i-1} \cup \left( \cup_{j=1}^{N_j} E^{x_{i,j} \frac{r_i}{N_{i}}} \right).  
\end{align*}

We claim that $K_0 = \cup_{i=0}^{\infty} \Gamma_{i}$ is the desired set. First note that since each $\Gamma_i$ is rectifiable, $K_{0}$ is rectifiable. Moreover, $\{x_{i, j}\}_{j=1}^{N_{i}} \subset \Gamma_{i-1}$ for all $i$ ensures $K_0$ inherits connectivity from $E_{\infty}$.  Furthermore, since $\{\Gamma_{i}\}$ is a nested sequence increasing to $K_{0}$ and $\sum_i r_i \le 1$, 

\begin{align*}
\mathcal{H}^1(K_{0}) & = \mathcal{H}^1 \left(E_{\infty} \cup \bigcup_{i= 1}^{\infty} \bigcup_{j = 1}^{N_i} E^{x_{i,j}, \frac{r_i}{N_i}} \right) \le \mathcal{H}^1(E_{\infty}) \left(1 + \sum_{i =1}^{\infty} r_i \right) \le 2\mathcal{H}^1(E_{\infty}).
\end{align*}

It only remains to show that for $x \in K_{0}$ and $\delta > 0$ that $C_{K_{0}}(x,\delta) = \infty$. To this end, fix $x \in K_{0},$ and $\delta>0$. By definition of $K_{0}$, there exists $\ell_{0}$ such that $x \in \Gamma_{\ell_{0}}$. Then, by the net property of the points $\{x_{i,j}\}$, it follows that for $\ell -1 \ge \ell_{0}$ large enough that $2^{-\ell -1} < \delta/4$, there exists $i \le \ell$ with $x_{i,j} \in \Gamma_{\ell-1} \cap B(x,\delta/2) \subset K_{0} \cap B(x,\delta/2)$. Writing $\mu = \cH^{1} \restr K_{0}$ and $\mu_{i,j} = \cH^{1} \restr E^{x_{i,j},\frac{r_{i}}{N_{i}}}$ it follows from monotonicity of the integral that
\begin{equation} \label{e:ceinf}
 \int_{B_{\delta}(x)} \int_{0}^{\delta} \beta_{K_{0};2}(y,r)^{2} \frac{dr}{r} d \mu(y)  \ge \int_{ B_{\delta/2}(x_{i,j})} \int_{0}^{\delta/2} \beta_{\mu_{i,j};2}(y,r) \frac{ dr}{r} d \mu_{i,j}(y),
\end{equation}
or equivalently $C_{K_{0}}(x,\delta) \ge C_{\mu_{i,j}}(x_{i,j}, \delta/2)$. Recalling that $E^{z,t} = tE_{\infty} + z$, we use \eqref{e:ceinf}, Proposition \ref{f:02}(2), and Lemma \ref{l:01} to conclude
$$
C_{K_{0}}(x, \delta) \ge C_{E^{x_{i,j}, \frac{r_{i}}{N_{i}}}}\left(x_{i,j}, \frac{\delta}{2}\right) = \frac{r_{i}}{N_{i}} C_{E_{\infty}}\left(0, \frac{ \delta N_{i}}{2 r_{i}} \right) = \infty.
$$ 
 Since $x \in K_{0}$ and $\delta > 0$ are arbitrary this finishes the proof.
\end{proof}

\begin{remark}
Since $K_{0}$ from Theorem \ref{t:K_0} is connected, $\cH^{1}(\overline{K_{0}}) = \cH^{1}(K_{0}) < \infty$ and $\overline{K_{0}}$ is compact, see \cite[Lemma 3.4, 3.5]{schul2007subsets}. Thus $\overline{K_{0}}$ is a rectifiable curve by Wazewski's theorem, see \cite[Lemma 3.7]{schul2007subsets} or \cite[Theorem 4.4]{alberti2017structure}.
\end{remark}
\section{Proof of Theorem \ref{t:A_0}}
To produce the desired set $A_{0}$, we use approximations of the $4$-corner Cantor set to produce a base set that has precise control on the $\beta$-numbers at the origin, then we carefully iterate this set ``on itself'' in order to preserve Ahlfors regularity.

\subsection{Approximations to the $4$-corner Cantor set}

Consider the following sequence of approximations to the $4$-corner cantor set, by sets of positive and finite $\cH^{1}$-measure.

Let $E_{0} = [0,1) \times \{0\}$ and inductively define
\begin{equation} \label{e:ek}
E_{k} = \sum_{(i,j) \in \{0,3\}^{2}} p_{ij} + 2^{-2} E_{k-1} \quad \text{where} \quad p_{ij} = \left( \frac{i}{2^{2}}, \frac{j}{2^{2}} \right).
\end{equation}

The word similarity is used to refer to any mapping that can be written as a composition of scalings, rotations, reflections, and translations. Throughout the rest of the paper, we say that two sets are similar if one is the image of the other by a similarity. In reality the similarities we discuss can always be written as a scaling and translation, as in \eqref{e:ek}.

We let $\Delta$ denote the collection of tetradic half-open cubes in $\R^{2}$, that is 
$$
\Delta = \{ [a 2^{-2k}, (a+1) 2^{-2k}) \times [b 2^{-2k}, (b+1) 2^{-2k}) \mid a,b, k \in \Z \}.
$$
For some $Q \in \Delta$, we let $\ell(Q)$ denote the sidelength of $Q$. We partition the tetradic cubes into cubes of fixed sidelength by defining $\Delta^{i} = \{ Q \in \Delta \mid \ell(Q) = 2^{-2i} \}$.

In general, for a set $E \subset \R^{2}$ we the \emph{length of $E$} and respectively \emph{height of $E$} by
$$
\ell(E) = \diam \{ \pi_{x}(E) \} \quad \text{and} \quad h(E) = \diam \{ \pi_{y}(E) \}
$$
where $\pi_{x}$ and $\pi_{y}$ denote the orthogonal projection onto the horizontal and vertical axes. In particular, for a cube $Q$, this notion of length coincides with its sidelength.

\begin{definition}[Clusters and sub-clusters] Any set  which is similar to any $E_{k}$ or $E_{k} \cup [0,1) \times \{0\}$ for $k \in \N$ will be called a \emph{cluster}. 

Moreover, for fixed $k \in \N$, we will call $E_{k}$ the \emph{$0$th sub-cluster of $E_{k}$} and the $2^{2k}$ line segments that make up $E_{k}$ are called the \emph{$k$th -subclusters of $E_{k}$}. For $\ell \in \{1, \dots, k-1\}$, the $2^{2\ell}$-sets contained in $E_{k}$ which are similar to $E_{k-\ell}$ are called the \emph{$\ell$th sub-clusters of $E_{k}$}.
\end{definition}

\begin{definition}[Root points]
We associate to each cluster and each cube a root point.  The \emph{root point} of a cluster $E$ is the lower-most and left-most point in the cluster. Since a sub-cluster is itself a cluster, the notion of a root point extends to sub-clusters. For a cluster $E$, we let $x_{E}$ denote its root point. For a tetradic cube $Q \in \Delta$ we let $x_{Q}$ denote the lower-most and left-most point of $Q$ and call $x_{Q}$ \emph{the root point of $Q$}.

\end{definition}

\begin{proposition} \label{p:ek}
For fixed non-negative integer $k$, the set $E_{k}$ has the following properties.
\begin{enumerate}
\item Each $E_{k}$ is a finite union of $2^{2k}$ intervals each of length $2^{-2k}$. In particular, $\cH^{1}(E_{k}) = 1$ and $E_{k}$ is countably $1$-rectifiable. Moreover, each connected component $I$ of $E_{k}$ has $\partial I \subset \ell(I) \Z^{2} = 2^{-2k} \Z^{2}$ and consequently is contained in a line $\R \times \{a 2^{-2k}\}$ for some $a \in \N_{0}$.
\item If $j \ge 0$ is an integer and if $Q \in \Delta^{j}$ is such that $Q \cap E_{k}$ is non-empty, then 
\begin{equation} \label{e:cubepic}
Q \cap E_{k} = 
\begin{cases}
x_{Q} + [0, \ell(Q)) \times \{0\} & j \ge k \\
x_{Q} + 2^{-2j}E_{k-j} & j \le k
\end{cases}
\end{equation}
\item Each $E_{k}$ is Ahlfors regular with regularity constant independent of $k$.
\item For $0 \le j \le k$ an integer, the $j$th subcluster of $E_{k}$ has $\cH^{1}$-measure $2^{-2j}$.
\item For $1 \le j \le k$ an integer, the $j$th subclusters of $E_{k}$ are $2 \cdot 2^{-2j}$-separated horizontally and at least $2 \cdot 2^{-2j}$-separated vertically.  In fact, they are $\left(3 - \frac{3}{4} \sum_{i=1}^{k-j} 2^{-2i} \right) \cdot 2^{-2j}$-separated vertically. 
\item If $J \subset E_{k}$ is a connected component, then $J$ is a vertical distance of $3 \cdot 2^{-2k}$ from the nearest connected component $J^{\prime}$ of $E_{k}$.
\item There exists a universal constant  $c > 0$ such that if $k\ge2$ and $\mu_{k} = \cH^{1} \restr E_{k}$, then for all $x \in E_{k}$, 
$$
\int_{6 \cdot 2^{-2k}}^{1} \beta_{\mu_{k}}(x,r)^{2} \frac{dr}{r} \ge c (k-2)
$$
\end{enumerate}
\end{proposition}

\begin{proof}
(1) follows immediately from \eqref{e:ek} since each $p_{ij} \in 2^{-2} \Z^{2}$.

To see (2), we first note that the case $j = 0$ is clear for any $k \in \N$. Further, the case $k = 0$ is clear for all $j \in \N$. To procede inductively suppose that \eqref{e:cubepic} holds for all $k \in \N$ when $j = \ell-1$. We will show it holds for all $k \in \N$ when $j = \ell$. Indeed, suppose that $Q \in \Delta^{\ell}$ has non-empty intersection with $E_{\ell}$. Let $x_{Q}$ be the root of $Q$. Choose $p \in \{p_{ij}\}_{(i,j) \in \{0,3\}^{2}}$ such that $Q \subset p + [0,2^{-2})^{2}$. Then,
$4(Q \cap E_{k} - p) = (4Q - 4p) \cap (4E_{k} - 4p) = \tilde{Q} \cap E_{k-1}$ where $\tilde{Q} \defeq 4Q - 4p \in \Delta^{\ell-1}$. By the inductive assumption,
\begin{align*}
\tilde{Q} \cap E_{k-1} & = \begin{cases}
x_{\tilde{Q}} + [0, \ell(\tilde{Q})) \times \{0\} & \ell-1 \ge k-1 \\
x_{\tilde{Q}} + 2^{-2(i-1)} E_{(\ell-1) - (i-1)} & \ell-1 \le k-1.
\end{cases}
\end{align*}
Translating and scaling this back to what this means about $Q \cap E_{k}$ verifies the induction. 

(3) follows from (1) and (2) since these imply that $\frac{ \cH^{1}(Q \cap E_{k})}{\ell(Q)} = 1$ for tetradic cubes $Q$ with $\ell(Q) \le 1$ that intersect $E_{k}$.  This suffices since any ball contains a tetradic cube of comparable sidelength and is contained in $4^{2}$ tetradic cubes of comparable sidelength.

(4) is equivalent to showing that $E_{k}$ is made of $2^{2k}$ intervals, each of length $2^{-2k}$.

(5) The horizontal separation is verified by an argument similar to the vertical separation. For the vertical separation, we only verify that the vertical separation is at least $2 \cdot 2^{-2j}$. Indeed, this follows since $E_{\ell}$ is contained in the horizontal strips $\R \times [0, 1/4] \cup [3/4, 1]$ for all $\ell$. Then, the scaling from \eqref{e:ek} ensures that the $j$th subclusters, which arise by applying \eqref{e:ek} $j$ times to the sets $E_{k-j}$ are vertically $2\cdot 2^{-2 j} = \frac{1}{2} 2^{-2(j -1)}$-separated. The reason the height-bound can be improved, is because the $j$th subclusters are actually contained in smaller strips. See for instance, $E_{1}$, where the first subclusters are contained in lines, and $E_{2}$ where the first subclusters are contained in the strips $\R \times [0, \frac{3}{16}] \cup [ \frac{12}{16}, \frac{15}{16}]$. 

(6) follows from the fact that vertically-closest connected components in $E_{k}$ come from the connected components of $E_{1}$ which are $3\cdot2^{-2}$ separated. After being scaled by $2^{-2}$ in \eqref{e:ek} another $(k-1)$ times the separation is reduced to a distance of $3 \cdot 2^{-2k}$ as claimed. This coincides with the precise formula in (5) and could be considered as a base case for induction on $j$ for the interested reader.

(7) Throughout the proof of (7), we fix integers $1 \le j < k$ and $k \ge 2$.

\underline{Claim 1:} For all $x \in E_{k}$ there exists some $x^{\prime} \in E_{j}$ with
\begin{equation} \label{e:smallscales}
\dist(x,x^{\prime}) \le 2^{-2j}
\end{equation}

\textit{Proof of Claim 1.} Note that the scaling in \eqref{e:ek} ensures that for some $\ell$, we know that every $x \in E_{\ell + 1}$ is within a distance $3 \cdot 2^{-2(\ell + 1)}$ of a point in $E_{\ell}$.  Iterating verifies the claim by showing for $x \in E_{k}$ there exists $x^{\prime} \in E_{j}$ such that
$$
\dist(x,x^{\prime}) \le \sum_{\ell=j+1}^{k} 3 \cdot 2^{-2\ell} \le 3 \sum_{\ell = j+1}^{\infty}  2^{-2 \ell} = 4 \cdot 2^{-2 (j+1)}.
$$

\underline{Claim 2:} There exists $c$ independent of $j$ such that for all $5 \cdot 2^{-2j} \le r \le 11 \cdot 2^{-2j}$ and all $x^{\prime} \in E_{j}$,
$$
\beta_{\mu_{j};2}^{1}(x^{\prime},r)^{2} \ge c
$$

\textit{ Proof of Claim 2.} Let $J \subset E_{j}$ be the connected component containing $x^{\prime}$. By (4)-(6) of this proposition, it follows that for $r \ge 5 \cdot 2^{-2j} = \sqrt{ \left(3 \cdot 2^{-2j} \right)^{2} + \left( 4 \cdot 2^{-2j} \right)^{2}}$, the ball $B_{r}(x^{\prime})$ contains $J$ and 3 other connected components of $E_{j}$. Consequently, there are two horizontal lines $L^{u}$ and $L^{d}$ such that $B_{r}(x^{\prime}) \cap \left( L^{u} \cup L^{d}\right)$ contains at least 4 connected components of $E_{j}$. Part (1) of this proposition ensures,
\begin{equation} \label{e:minline}
\min \{ \mu_{j} \left(L^{u} \cap B_{r}(x^{\prime}) \right), \mu_{j} \left( L^{d} \cap B_{r}(x^{\prime}) \right) \} \ge 2 \cdot 2^{-2j}.
\end{equation}
Moreover, part (6) ensures that the distance between $L^{u}$ and $L^{d}$ is $3 \cdot 2^{-2j}$, which combined with \eqref{e:minline} forces that any line $L$ satisfies,
\begin{equation} \label{e:massbound}
\mu_{j} \left( \left\{ y \in B_{r}(x^{\prime}) \big| \dist (y,L) \ge 3 \cdot 2^{-2j-1} \right\} \right) \ge 2 \cdot 2^{-2j}.
\end{equation}
Finally, recalling $5 \cdot 2^{-2j} \le r \le 11 \cdot 2^{-2j}$, \eqref{e:massbound} implies

$$
\inf_{L} \int_{B_{r}(x^{\prime})} \left( \frac{ \dist(y,L) }{r} \right)^{2} \frac{ d \mu_{j}(y)}{r} \ge \left( \frac{ \frac{3 \cdot 2^{-2j}}{2}}{r} \right)^{2} \left( \frac{2\cdot2^{-2j}}{r} \right) \ge  c
$$
which verifies Claim 2.

\underline{Claim 3:} There exists $c^{\prime}$ such that for all $x \in E_{k}$ and all integers $1 \le j < k$ and $\rho$ such that $6 \cdot 2^{-2j} \le \rho \le 12 \cdot 2^{-2j}$,
\begin{equation} \label{e:mukbeta}
\beta_{\mu_{k};2}^{1}(x,\rho)^{2} \ge c^{\prime}.
\end{equation}

\textit{Proof of Claim 3.} Claim 1 ensures that for all $5 \cdot 2^{-2j} \le r \le 11 \cdot 2^{-2j}$ there exists $x^{\prime} \in E_{j}$ such that $B_{r}(x^{\prime}) \subset B_{\rho}(x)$. As in Claim 2, fix lines $L^{d}$ and $L^{u}$ such that $B_{r}(x) \cap \left( L^{u} \cup L^{d} \right)$ contains at least $4$ connected components of $E_{j}$. Choose $a$ so that $L^{d} = \R \times \{a\}$ and $L^{u} = \{a + (0,3 \cdot 2^{-2j})\} + \R \times \{0\}$ . Moreover, suppose the left-most connected component of $L^{u}$ has right-most endpoint with $x$-value equal to $c_{1}$. Define $L_{v} = \{c_{1} + 2^{-2j} \} \times \R$ and $L_{h} = a + 2^{-2j}$. By Proposition \ref{p:ek}(5,6), the neighborhoods $N_{v} = B_{2^{-2j}}(L_{v})$ and $N_{h} = B_{2^{-2j}}(L_{h})$ are disjoint from $E_{\ell}$ for all $\ell \ge j$. See Figure \ref{E13}.

Consequently, for any line $L$ the nieghborhood $B_{2^{-2j-1}}(L)$ can intersect at most $4$ of the ``quadrants'' made by the neighborhoods of $N_{v}$ and $N_{L}$. Making a generous estimate since the ball may cut-off part of one of the quadrants in Figure \ref{E13}, we conclude
\begin{equation} \label{e:01}
\mu_{k} \left( \{ y \in B_{r}(x^{\prime})  \mid \dist(y,L) \ge 2^{-2j-2} \} \right) \ge 2^{-2j-2}
\end{equation}
where the measure-bound comes Proposition \ref{p:ek}(1). Since $B_{r}(x^{\prime}) \subset B_{\rho}(x)$ and $1 \le \frac{\rho}{r} \le C < \infty$ Claim 3 follows from \eqref{e:01} analogously to how Claim 2 followed from \eqref{e:massbound}.

Finally, we verify (7) because
$$
\int_{6 \cdot 2^{-2j}}^{1} \beta_{\mu_{k}}(x,\rho)^{2} \frac{d \rho}{\rho} \ge \sum_{j=2}^{k}  \int_{6 \cdot 2^{-2j}}^{11 \cdot 2^{-2j}} c^{\prime} \frac{ d \rho}{\rho} = c (k-2)
$$

%
%
%
%
%
\end{proof}

\begin{figure}
\includegraphics[width=.9 \textwidth]{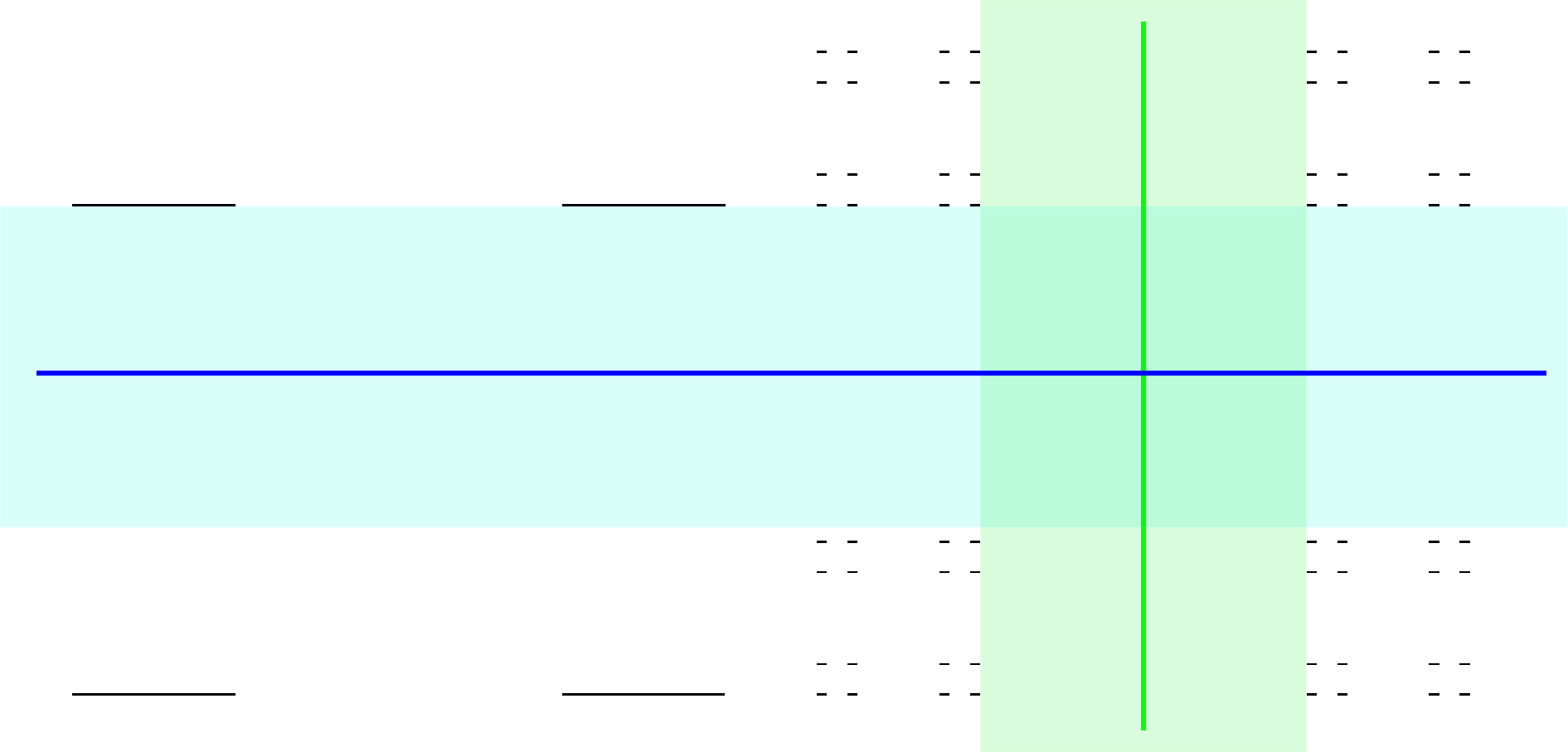} 
\caption{When $j = k-2$, the picture displays a subclusters of equal length for $E_{j}$ and $E_{k}$ on the left and right respectively. In $E_{k}$, the line $L_{v}$ and its neighborhood $N_{v}$ are in green, whereas the line $L_{h}$ and its neighborhood $N_{h}$ are drawn where it would pass through both $E_{j}$ and $E_{k}$} 
\label{E13}
\end{figure}

We construct $\Sigma_{0}$ from approximations to the $4$-corner Cantor set by first defining
\begin{equation} \label{e:en} 
E(n) = (2^{-2n},0) + 2^{-2n} E_{2^{2n}} \quad \text{and} \quad \Sigma_{0} = \bigcup _{n} E(n) \cup \left( [0,1) \times \{0\} \right).
\end{equation}

\begin{proposition} \label{p:sig0props}
$\Sigma_{0}$ has the following properties.
\begin{enumerate}
\item $0 < \cH^{1}(\Sigma_{0}) < \infty$ and $\Sigma_{0}$ is countably $1$-rectifiable.
\item If $j \ge 0$ is an integer and $Q \in \Delta^{j}$ is such that $Q \cap \Sigma_{0} \neq \emptyset$, then 

\begin{equation} \label{e:sig0pic}
Q \cap \Sigma_{0} = 
\begin{cases}
\Sigma_{0} \cap [0,\ell(Q))^{2}  & x_{Q} = (0,0) \\
x_{Q} + 2^{-2j}E_{k} \text{ for some $k$ } & x_{Q} \neq (0,0) \text{ and } \pi_{y}(x_{Q}) \neq 0 \\
x_{Q} + 2^{-2j} E_{k} \cup [0,\ell(E_{k})) \times \{0\}  & x_{Q} \neq (0,0) \text{ and } \pi_{y}(x_{Q}) = 0.
\end{cases}
\end{equation}
\item $C_{\Sigma_{0}}(0,\delta) = + \infty$ for all $\delta > 0$.
\end{enumerate}
\end{proposition}

\begin{proof}
(1) $\Sigma_{0}$ has positive and finite mass due, Proposition \ref{p:ek}(1) and the geometric scaling in $\eqref{e:en}$. It is also the countable union of countably $1$-rectifiable sets by Proposition \ref{p:ek}(1).

(2) The case when $x_{Q} = (0,0)$ is clear. Suppose $x_{Q} \neq (0,0)$. There exists unique $a,b$ such that
\begin{equation}
x_{Q} = \left( a 2^{-2j}, b 2^{-2j} \right).
\end{equation}
If $j = 0$, $Q \cap \Sigma_{0} \neq \emptyset$, and $\Sigma_{0} \subset [0,1)^2$ forces $a=b=0$. Therefore, $j\ge1$. Since $h(E_{2^{2n}}) < \ell(E_{2^{2n}})$ and the $E(n)$ only use a translation in the positive horizontal direction of $E_{2^{2n}}$ and a homogeneous scaling, it follows that $\Sigma_{0} \cap Q \neq \emptyset$ implies $0 \le b < a$ so that $a \ge 1$. Since, $\ell(Q) = 2^{-2j}$ it follows that $a 2^{-2j} \ge \ell(Q)$. Comparing the translation and scaling sizes in \eqref{e:ek}, $a \ge 2^{2j} \ell(Q)$ implies
\begin{equation} \label{e:02}
\Sigma_{0} \cap Q = 
\begin{cases}
Q \cap E(n) & b \ge 1\\
Q \cap \left(E(n) \cup [0,\ell(E(n))) \times \{b\}\right) & b = 0 
\end{cases}
\end{equation}
for some specific $n \le j$. For simplicity of writing, assume we're in the first case. Then, $2^{2n} \left( Q \cap E(n) - (2^{-2n},0) \right) =\left(2^{2n} \left( Q - (2^{-2n}, 0) \right) \right) \cap E_{2^{2n}}$ or equivalently
\begin{equation} \label{e.1}
Q \cap E(n) = (2^{-2n},0) + 2^{-2n} \left( 2^{2n} \left(Q - \left(2^{-2n},0 \right) \right) \cap E_{2^{2n}} \right).
\end{equation}

\noindent In light of \eqref{e.1}, it follows that \eqref{e:cubepic} implies the 2nd case of \eqref{e:sig0pic} since $2^{2n}(Q - (2^{-2n},0)) \in \Delta^{j-n}$ and $n \le j$. Analogously the $b=0$ case corresponds to the 3rd case of \eqref{e:sig0pic}.

(3) Fix $\delta > 0$. Choose $N$ so that $11 \cdot 2^{-2N} < \delta/2$, so that for all $n \ge N$, $E(n) \subset B_{\delta}(0)$. Then, with $\mu = \cH^{1} \restr \Sigma_{0}$ and $\mu_{n} = \cH^{1} \restr E(n)$, it follows from Proposition \ref{p:ek} (1,7), Proposition \ref{f:02} (2), and the scaling in \eqref{e:en} that
\begin{align*}
C_{\Sigma_{0}}(0, \delta) 
 \ge \sum_{n \ge N} \int_{E(n)} \int_{0}^{2^{-2n}} \beta_{\mu_{n};2}^{1}(x,r)^{2} \frac{dr}{r} d \mu_{n}(x) 
 \ge \sum_{n \ge N} c (2^{2n}-2) \cH^{1}(E(n)),
\end{align*}
which diverges and completes the proof.
\end{proof}

We wish to iterate $\Sigma_{0}$ densely along itself while being careful to maintain Ahlfors upper- and lower-regularity. This is attained by scaling, and being careful where we iterate.

\begin{definition}[Tail points]
We say a point $y$ is a \emph{tail point of $E$} if $0 < \cH^{1}(E) < \infty$ and there exists a tetradic number $r$ and $\delta > 0$ such that
$$
y+ r \Sigma_{0} \cap B_{\delta}  \subseteq  E.
$$
Note, if $y \in B_{\delta}(x)$ is a tail point of a set $E$, then $C_{E}(x, \delta) \equiv \infty$. See Claim 1 of Theorem \ref{t:A_0}.
\end{definition}

\begin{definition}[Iterative construction] \label{d:sigmaic}
Let $\Sigma_{0}$ be as above. Supposing that $\Sigma_{i-1}$ has been defined, we define a (possibly empty) special collection of tetradic points, 
\begin{equation} \label{e:di}
D^{i} = \left \{ x \in 2^{-2i} \Z^{2}\bigg| \left( x + [0, 2^{-2i})^{2} \right) \cap \Sigma_{i-1} = x + [0, 2^{-2i}) \times \{0\} \right\},
\end{equation} 
and define $\Sigma_i$ by
\begin{equation} \label{e:sigmai}
\Sigma_{i} = \Sigma_{i-1} \bigcup \left\{ \cup_{x \in D^{i}} x + 2^{-8i} \Sigma_{0} \right\}.
\end{equation}

Define,
\begin{equation} \label{e:A0}
A_{0} = \cup_{j \in \N} \Sigma_{j}.
\end{equation}
\end{definition}

\begin{proposition} \label{p:sigmaiprops}
The sets $\{\Sigma_{j}\}_{j=0}^{\infty}$ and $\{D^{j}\}_{j=1}^{\infty}$ as in Definition \ref{d:sigmaic} have the following properties:

\noindent(1) $\Sigma_{j-1} \subset \Sigma_{j}$ for all $j \ge 1$,

\noindent(2) $\Sigma_{j}$ is contained in countably many horizontal line segments with tetradic heights.

\noindent(3) $D^{j}$ is non-empty infinitely often.

\noindent(4) If $I$ is a connected component of $\Sigma_{j}$ then $\partial I \subset \ell(I) \Z^{2}$.

\noindent(5) $\Sigma_{j}$ contains no connected component of length at least $2^{-2j}$ that contain no tail point.
\end{proposition}

\begin{proof}
Indeed, (1) follows from \eqref{e:sigmai}. 

(2) Follows by induction. For $\Sigma_{0}$ it follows from Proposition \ref{p:ek} (1) combined with the scaling in \eqref{e:en}. For general $\Sigma_{j}$ induction holds due to the fact that each scaled copy of $\Sigma_{0}$ in \eqref{e:sigmai} has a tail point on the dyadic lattice $D^{i}$ which is coarser than the tetradic scaling factor of $\Sigma_{0}$. 

(3) folows from (2). (5) follows from (4) and the definition of $D^{j}$ in \eqref{e:di}.

(4) If $I$ is a connected component of $\Sigma_{j}$ then there exists $y \in D^{i}$ some $i \le j$ such that
$I$ is a connected component of $y + 2^{-8i}\Sigma_{0}$. But then, $2^{8i}(I - y)$ is a connected component of $\Sigma_{0}$. Since $y \in 2^{-2i} \Z^{2}$, Propositions \ref{p:ek}(1) and \ref{p:sig0props}(2) ensure $\partial \left(2^{8i}(I-y) \right) \in 2^{8i} \ell(I) \Z^{2}$ which verifies (4).

\end{proof}

\begin{definition}[Associated cubes]
Any cluster (or subcluster) $E$ has associated to it the dyadic cube $Q_{E} = x_{E} + [0, \ell(E))^{2}$. In particular, by Proposition \ref{p:ek} (5) it follows that if clusters $E,E^{\prime}$ are disjoint with $\ell(E) = \ell(E^{\prime})$, then $Q_{E}, Q_{E^{\prime}}$ are disjoint cubes. Moreover, for some cluster $E$, the root point of $Q_{E}$ and the root point of $E$ coincide. 
\end{definition}

\begin{definition} \label{d:cubefamily}
We associate to the base set $\Sigma_{0}$ the following family of cubes
\begin{align} \label{e:cubessig0}
\cQ_{\Sigma_{0}} &= \big\{ [0,2^{-2i})^{2} : i \ge 0 \big\} \cup \left\{ Q_{E} : E \text{ is a subcluster of $E(n) \subset \Sigma_{0}, n \ge 1$ } \right\} 
\end{align}

By similarity, for any $y \in D^{i}$ we associate to $y + 2^{-8i} \Sigma_{0}$ the family of cubes
\begin{equation} \label{e:cubesy}
\cQ_{y} = \left( y + 2^{-8i} \cQ_{\Sigma_{0}} \right) \bigcup \left( y + \{ [0,2^{-2k})^{2} : i \le k  \} \right).
\end{equation}

We will let
\begin{equation} \label{e:cQ}
\cQ = \cup_{i \ge 0} \cup_{y \in D^{i}} \cQ_{y}
\end{equation}
which we stratify by scale in the following sense
\begin{equation} \label{e:cQi}
\cQ^{i} = \{ Q \in \cQ \mid \ell(Q) = 2^{-2i} \} 
\end{equation}
and we enumerate the elements $\cQ^{i}$ so that
\begin{equation} \label{e:cQij}
\cQ^{i} = \{ Q_{j}^{i} \}_{j=1}^{N(i)}.
\end{equation}
Finally, for $Q \in \cQ$ and any positive integer $\ell$ we let $\cC_{\ell}(Q) = \{ Q^{\prime} \in \cQ \mid \ell(Q^{\prime}) = 2^{-2\ell} \ell(Q) \}$, and call $\cC_{\ell}(Q)$ the \emph{$\ell$th descendent cubes of $Q$}.
\end{definition}

\begin{lemma}\label{l:cubeclass}
For all $i \ge 0$ and all cubes, $Q^i_j \in \mathcal{Q}^i,$ $\Sigma_i \cap Q^i_j$ is similar to one of the following:
\begin{enumerate}
\item $\left(2^{-2k}\Sigma_0 \cup [0, 1) \times \{0\} \right) \cap [0, 1)^2$ for some integer $k$.
\item $E \cap Q_E$ for some sub-cluster $E \subset E(n)$ for some integer $n \ge 1$
\end{enumerate}
\end{lemma}

This follows immediately from the explicit definition of cubes. 

\begin{lemma} \label{l:cubesRtet} $\cQ^{j} \subset \Delta^{j}$ and for all $Q \in \Delta^{j}$, then either $Q \cap \Sigma_{j} = \emptyset$ or $Q \in \cQ_{j}$.

\end{lemma}

This follows from an induction argument similar to the proofs of Propositions \ref{p:ek} (1) and \ref{p:sig0props} (2). The key observation in the induction is that the scaling in \eqref{e:sigmai} ensures that all tail points added in the $j$th stage have root points in tetradic lattices that are coarser than the length of the scaled copy of $\Sigma_{0}$ being added.

\begin{corollary} \label{c:cubesgood} The cubes $\cQ$ have the following nice properties:
\begin{enumerate}
\item Each collection $\cQ_{i}$ is a disjoint collection of cubes, and for any $Q \in \cQ$ and any integer $\ell \ge 0$, $\cC_{\ell}(Q)$ is a disjoint collection of subcubes of $Q$.
\item For all non-negative integers $i$ and $j$,
\begin{equation} \label{e:gencon}
\Sigma_{i} \subseteq \cup_{Q \in \cQ_{j}} Q
\end{equation}
\item In particular, for any $Q_{0} \in \cQ_{i}$
\begin{equation} \label{e:contain}
\Sigma_{i} \cap Q_{0} = \Sigma_{i} \bigcap \left( \cup_{Q \in \cC_{1}(Q)} Q \right)
\end{equation}

\end{enumerate}
\end{corollary}


\begin{proof}[Proof of Theorem \ref{t:A_0}]
Indeed, by Lemma \ref{p:sig0props} (1), $\Sigma_{0}$ is $1$-rectifiable, and $A_{0}$ is a countable union of scaled translations of $\Sigma_{0}$ so $A_{0}$ is $1$-rectifiable .

Next, we show that $A_{0}$ is $1$-Ahlfors regular.  Indeed, it suffices to show that there exists $0 < c \le C < \infty$ independent of $i$ such that for for any $j \ge 0$, $Q \in \Delta^{j}$, and $Q \cap A_{0} \neq \emptyset$,
\begin{equation} \label{e:denso}
 c \ell(Q) \le \cH^{1}(Q \cap A_{0}) \le C \ell(Q).
\end{equation}

We do this by showing similar bounds for $\frac{\cH^{1}(Q \cap \Sigma_{j})}{\ell(Q)}$ for cubes $Q \in \Delta^{j}$ that intersect $\Sigma_{j}$, and then proving that not too much additional mass is added to the cube $Q$. 

Due to Lemma \ref{l:cubesRtet} the condition that $Q \in \Delta^{j}$ and $Q \cap A_{j} \neq \emptyset$ is equivalent to $Q \in \cQ_{j}$. Since $Q \in \cQ_{j}$ Lemma \ref{l:cubeclass} characterizes what $Q \cap \Sigma_{j}$ looks like and we conclude 
\begin{equation} \label{e:densi}
\ell(Q) \le \cH^{1}(Q \cap \Sigma_{j}) \le 3 \ell(Q),
\end{equation}
by considering each of the three cases in Lemma \ref{l:cubeclass}. Indeed, each cube either contains its entire bottom portion, or contains a cluster $E$ with $\ell(E) = \ell(Q)$. In either case this implies the lower bound in \eqref{e:densi}. On the other hand, we know that a rough upper-bound is to assume that $Q \cap \Sigma_{j}$ contains a cluster with a line segment at the bottom, and contains $\Sigma_{0}$ scaled by $2^{-2k}$, then by Proposition \ref{p:ek}, the upper bound in \eqref{e:densi} follows.

It remains to show that \eqref{e:densi} implies \eqref{e:denso}. Due to Proposition \ref{p:sigmaiprops} (1), the lower-bound in \eqref{e:denso} is inherited directly from \eqref{e:densi} . The upper-bound follows with the additional observation that for $\ell \ge j$,
$$
\cH^{1} \left( Q \cap \Sigma_{\ell+1} \setminus \Sigma_{\ell} \right) \le \#|D_{\ell+1}| 2^{-8 (\ell+1)} \cH^{1} \left(\Sigma_{0} \right) \le  2^{-4(\ell+1)} \cH^{1}(\Sigma_{0}).
$$
Summing over $\ell \ge j$ verifies \eqref{e:denso}. It is a standard argument to go from Ahlfors regularity in tetradic/dyadic cubes to in balls, see for instance the brief description in the proof of Proposition \ref{p:ek}(3). Since the cubes in $\cQ$ are all the tetradic cubes with non-empty intersection with $A_{0}$, we have regularity in tetradic cubes.

Finally, to see that $C_{A_{0}}(x, \delta) = \infty$ it suffices to show the following claim.

\noindent \underline{Claim 1-} If $x \in A_{0}$ and $\delta > 0$, then there is a tail point in $A_{0} \cap B_{\delta/2}(x)$. 

Briefly assuming that Claim 1 holds, the fact that $C_{A_{0}}(x,\delta) = \infty$ for all $x \in A_{0}$ and $\delta > 0$ follows since if $y$ is the tail point in $B_{\delta/2}(x)$ then, by Proposition \ref{p:sig0props} (3) and monotonicity of integrals of non-negative functions:
$$
C_{A_{0}}(x, \delta) \ge C_{A_{0}}(y, \delta/2) \ge C_{\Sigma_{0}}(0, \epsilon_{y}) = \infty,
$$
where $\epsilon_{y} > 0$ is some scale dependent on which $D^{i}$ the tail point $y$ is in.

\noindent To verify Claim 1, fix $x$ and $\delta$ as in the claim. Adopting the convention that $\Sigma_{-1} = \emptyset$ fix $i_{0}$ such that $x \in \Sigma_{i_{0}} \setminus \Sigma_{i_{0}-1}$. Choose $k$ to be the smallest natural number such that $\diam \left( 2^{-8k} \Sigma_{0} \right) \le \delta/4$.

\underline{Case 1}- $B_{\delta/4}(x) \cap \Sigma_{k}$ contains a tail. Since $\Sigma_{k} \subset A_{0}$ in this case the claim holds.

\underline{Case 2}- Otherwise, choose $k_{0} \ge k$ such that
$$
\begin{cases}
\left( \Sigma_{k_{0}-1} \setminus \Sigma_{k} \right) \cap B_{\delta/4}(x) = \emptyset \\
\left( \Sigma_{k_{0}} \setminus \Sigma_{k} \right) \cap B_{\delta/4}(x) \neq \emptyset,
\end{cases}
$$
that is $k_{0}$ is the first stage after $k$ where something new is added to the ball $B_{\delta/4}(x)$. The way something new is added to the ball $B_{\delta/4}(x)$ in the $k_{0}$th stage is if there exists $y$ such that,
$$
\{ y + 2^{-8k_{0}} \Sigma_{0} \} \cap \{ \Sigma_{k_{0}} \cap B_{\delta/4}(x) \}\neq \emptyset.
$$ But then, $y$ is a tail point of $\Sigma_{k_{0}}$ and consequently of $A_{0}$. By our choice of $k$, we conclude
$$
|x-y| < \diam( 2^{-4k_{0}} \Sigma_{0} ) + \delta/4 \le \delta/2.
$$
Hence the tail point $y$ is indeed in $B_{\delta/2}(x)$. So, by Proposition \ref{f:02}(2) 
$$
C_{A_{0}}(x,\delta) \ge C_{A_{0}}(y, \delta/2) \ge c C_{\Sigma_{0}}(0, \delta^{\prime}) = \infty.
$$ This completes the theorem.
\end{proof}

\bibliographystyle{amsplain}
\bibdata{references}
\bibliography{references}

\end{document}